\documentclass[12pt]{article}
\usepackage{amsmath,amssymb,amsfonts,amsthm}
\usepackage[margin=1in]{geometry}

\usepackage{url}
\usepackage{graphicx}
\usepackage{subcaption}
\usepackage{authblk}

\theoremstyle{plain}
\newtheorem{theorem}{Theorem}[section]
\newtheorem{proposition}[theorem]{Proposition}
\newtheorem{lemma}[theorem]{Lemma}

\title{Jungerman ladders and index 2 constructions for genus embeddings of dense regular graphs}
\author[1, 2]{Timothy Sun}\date{}

\affil[1]{\small Department of Computer Science, Columbia University}
\affil[2]{\small Department of Computer Science, San Francisco State University}

\begin{document}

\maketitle

\begin{abstract}
We construct several families of minimum genus embeddings of dense graphs using index 2 current graphs. In particular, we complete the genus formula for the octahedral graphs, solving a longstanding conjecture of Jungerman and Ringel, and find triangular embeddings of complete graphs minus a Hamiltonian cycle, making partial progress on a problem of White. Index 2 current graphs are also applied to various cases of the genus of the complete graphs, in some cases yielding simpler solutions, e.g., the nonorientable genus of $K_{12s+8}-K_2$. In addition, we give a topological proof of a theorem of Jungerman that shows that a symmetric type of such current graphs might not exist roughly ``half of the time.'' 
\end{abstract}

\section{Introduction}

Current graphs are combinatorial tools that describe covering space constructions for graph embeddings. They were first invented in the context of the Map Color Theorem~\cite{Ringel-MapColor}, where they were used to generate triangular embeddings of complete or near-complete graphs. While some of the simpler cases of the Map Color Theorem have since been resolved using other approaches (e.g. a recursive surgery method in Bonnington \emph{et al.} \cite{Bonnington-Exponential} or a block-design approach in Grannell \emph{et al.} \cite{Grannell-SurfaceEmbeddings}), current graphs remain the only known method for constructing genus embeddings of the complete graphs $K_n$ for some residue classes of $n$ modulo 12. 

Current graphs are themselves embedded (labeled) graphs, and the \emph{index} of a current graph refers to the number of faces in the embedding. Our focus here is on \emph{index 2} current graphs. Ringel and Youngs~\cite{RingelYoungs-German} were partially successful in finding genus embeddings of the complete graphs $K_n$ when $n = 12s+2$ or $12s+9$, using two similar families of index 2 current graphs. Unfortunately, their solution worked only for half of all positive integers $s$. Except to solve the single case $n = 14$, index 2 current graphs are absent in the original proof of the Map Color Theorem presented in Ringel~\cite{Ringel-MapColor}. 

Jungerman investigated these same two residue classes using \emph{orientable cascades}, which are index 1 current graphs embedded in a nonorientable surface that generate orientable embeddings. Orientable cascades can be interpreted as index 2 current graphs in orientable surfaces by considering the embedded surface's orientable double cover. For $n = 12s+9$, he encountered the same parity issue \cite[\S 8.4]{Ringel-MapColor}, but eventually found an even simpler solution that worked for all $s \geq 0$ \cite[\S 6.5]{Ringel-MapColor}. In his Ph.D.\ thesis \cite{Jungerman-OrientableCascades}, he explained why the aforementioned approaches failed to find the appropriate current graphs for all $s$ by proving an impossibility result for orientable cascades whose derived embeddings are of complete graphs. We provide a more modern, topological proof of a slight generalization of Jungerman's theorem. 

The primary success story of orientable cascades is Jungerman and Ringel's \cite{JungermanRingel-Octa} construction for triangular embeddings of the octahedral graphs $K_{2m}-mK_2$, $m \not\equiv 2 \pmod{3}$. We resolve the missing cases, the octahedral graphs whose genus embeddings are not triangular, mostly using orientable cascades, as well. Unlike the complete graphs, octahedral graphs do not suffer from the aforementioned parity issue. If we interpret complete graphs of even order and octahedral graphs as Cayley graphs on a cyclic group $\mathbb{Z}_{2m}$, then Jungerman's theorem essentially states that the culprit is the order 2 element---complete graphs have it in their generating set, but octahedral graphs do not. 

Constructing an infinite family of current graphs is a difficult task, so one common technique is to start with a ``ladder'' that can grow arbitrarily long. Jungerman's technique for orientable cascades, which we call a \emph{Jungerman ladder}, can be transferred to index 2 current graphs by simulating an ``orientable double cover'' of just the ladder. For all the remaining graphs that we consider, which include some complete graphs $K_{12s+k}$ and Hamiltonian cycle complements $K_{12s}-C_{12s}$, we find a family of orientable cascades for half of all values $s$, and a different family of index 2 current graphs for the other half.

The results in this paper were discovered prior to the author obtaining a copy of Jungerman's thesis \cite{Jungerman-OrientableCascades}, and much to the chagrin of the author, many of the results presented here regarding embeddings of the complete graphs were already proven there, including the aforementioned impossibility result. Our current graphs for $K_{13}$, $K_{12s+4}$, $s$ even, and $K_{12s+2}-K_2$, $s$ odd, are identical or nearly identical to Jungerman's, but our other constructions and impossibility proof are either simpler than those in the current literature (\cite{Jungerman-KnK2, Jungerman-OrientableCascades, Korzhik-Case8}) or completely new. Finally, as a byproduct of working with embeddings of octahedral graphs, we show, in Appendix \ref{app-k18}, how a genus embedding of $K_{18}$ can be obtained from a current graph of Jungerman and Ringel \cite{JungermanRingel-Octa}.

\section{Graph embeddings in surfaces}

We assume familiarity with graph embeddings and the theory of current graphs. For more information, see Gross and Tucker~\cite{GrossTucker}, especially \S3.2, and Ringel~\cite{Ringel-MapColor}.

Let $G$ be an undirected graph with vertex set $V(G)$ and edge set $E(G)$. We arbitrarily orient each of the edges in $E(G)$ so that each edge induces two arcs $e^+$ and $e^-$ with opposite orientations. We write $E(G)^+$ to denote the set of these arcs. The arcs incident with a given vertex $v \in V(G)$ are the ones directed out of $v$, and a \emph{rotation} of $v$ is a cyclic permutation of its incident arcs. A \emph{(pure) rotation system} of $G$ is an assignment of rotations to each of its vertices. An embedding $\phi\colon G \to S$ of $G$ in some surface $S$ is \emph{cellular} if the complement of its image $S \setminus \phi(G)$ decomposes into a disjoint union of disks $F(G,\phi)$, which we call \emph{faces}. Pure rotation systems are in correspondence with cellular embeddings of graphs in orientable surfaces up to orientation-preserving equivalence of embeddings, where the boundaries of the faces can be traced out from interpreting rotations as, say, clockwise orderings of the arcs with respect to some orientation of the surface $S$. 

A \emph{(general) rotation system} extends the above combinatorial description to nonorientable surfaces by including an \emph{edge signature} $\lambda\colon E(G) \to \{-1, +1\}$. We call an edge \emph{twisted} if it has signature $-1$, and \emph{normal}, otherwise. While traversing a face boundary, if we encounter a twisted edge, the local orientation of the walk becomes reversed, and we choose all future outgoing arcs based on the \emph{counterclockwise} ordering at the vertex until we encounter another twisted edge. In this language, a rotation system is pure if no edge is twisted. Different general rotation systems may be equivalent to one another, and even ones with twisted edges may correspond to an embedding in an orientable surface. A \emph{vertex flip} reverses a vertex's rotation and switches the signature of all of its incident edges. Two rotation systems are said to be equivalent if one can be reached from the other by some sequence of vertex flips.

Given an embedding and an orientation of each of its face boundaries, an edge is said to be \emph{bidirectional} if the two times a face boundary traverses that edge are in opposite directions, and \emph{unidirectional}, otherwise. For an orientable embedding, orienting the face boundaries consistently with some orientation of the surface makes every edge bidirectional.

Embeddings in surfaces are governed by the \emph{Euler polyhedral formula} 
$$|V(G)|-|E(G)|+|F(G,\phi)| = \chi(S),$$
where $\chi(S)$ denotes the \emph{Euler characteristic} of the surface. The Euler characteristic is $2-2g$ for the orientable surface of genus $g$ and $2-k$ for the nonorientable surface of genus $k$. The \emph{minimum orientable (resp. nonorientable) genus} $\gamma(G)$ (resp. $\overline{\gamma}(G)$) of a graph $G$ is the minimum genus over all embeddings in orientable (resp. nonorientable) surfaces. 

If the graph is simple and not $K_2$, then in any embedding, the length of every face is at least 3. Using this fact in conjunction with the Euler polyhedral formula yields the \emph{Euler lower bound} 
$$\gamma(G) \geq \frac{|E(G)|-3|V(G)|+6}{6} \,\,\text{~and~}\,\, \overline{\gamma}(G) \geq \frac{|E(G)|-3|V(G)|+6}{3}$$
on the genus of a graph. Equality is attained when all the faces are triangular, and both inequalities can be sharpened by applying the ceiling function to the right hand side. 

Let $\mathbb{Z}_n$ denote the integers modulo $n$, and let $S$ be a subset of $\{1, 2, \dotsc, \lfloor n/2\rfloor\} \subseteq \mathbb{Z}_n$. We define the \emph{circulant graph} $C(n,S)$ to be the simple graph with vertex set $\mathbb{Z}_n$ and edge set $\{(u,v) \mid |u-v| \in S\}$. In this paper, we consider the following families of circulant graphs:
\begin{itemize}
\item The \emph{complete graphs} $K_n \cong C(n,\{1,\dotsc,\lfloor n/2\rfloor\})$.
\item The \emph{octahedral graphs} $O_{2m} \cong C(2m,\{1,\dotsc,m-1\})$. 
\item The \emph{Hamiltonian cycle complements} $K_n-C_n \cong C(n,\{2,\dotsc,\lfloor n/2\rfloor\})$.
\end{itemize}
For the second of these classes, we deviate from usual convention (which would refer to them as ``$O_m$'') to emphasize the number of vertices rather than the ``dimension'' of the octahedron. 

We are also interested in another family of graphs that are not circulant graphs. Let $B_{2n+1}$ denote the graph with vertex set $0, 1, 2, \dotsc, 2n, x_0, x_1$, where all the numbered vertices are adjacent to one another, $x_0$ is adjacent to all the even vertices, and $x_1$ is adjacent to all the odd vertices. We call these graphs \emph{balanced split-complete graphs}. 

When the graph is nearly complete, one might expect that its genus equals its Euler lower bound, and the Map Color Theorem of Ringel, Youngs, \emph{et al.}\ shows that this is true for all complete graphs except the nonorientable genus of $K_7$. Our purpose is to provide new constructions of minimum genus embeddings for some of the graphs belonging to the aforementioned families. 

\subsection{Current graphs}

A \emph{current graph} $(\phi, \alpha)$ consists of a graph embedding $\phi\colon G \to S$ and a labeling $\alpha\colon E(G)^+ \to \Gamma$ of the graph's arcs with elements from a group $\Gamma$. We call $\Gamma$ the \emph{current group} and the arc labels \emph{currents}. If $\lambda\colon E(G) \to \{-1,+1\}$ is the edge signature, the edge labeling $\alpha$ needs to satisfy $\alpha(e^+) = -\lambda(e)\alpha(e^-)$ for every edge.

In all of our current graphs, we take $\Gamma$ to be a cyclic group $\mathbb{Z}_{2m}$ of even order, so it is meaningful to refer to the even and odd elements of this group. We draw our current graphs in the plane, possibly with edge crossings. If the vertex is solid (resp. hollow), then the rotation at that vertex is the clockwise (resp. counterclockwise) ordering of the arcs incident with that vertex. Because of the above requirement for the arc-labeling on twisted edges, they are drawn with two arrows pointing in opposite directions. For current graphs with twisted edges, when a corner of a face is visited while the local orientation of the corresponding face-boundary walk is reversed, we label it with a dot. We call such a corner \emph{reversed}.

The \emph{index} of a current graph is the number of faces in its embedding. We restrict ourselves to index 1 and 2 current graphs. The face boundaries are called \emph{circuits}---for index 1 current graphs, the circuit is labeled $[0]$, and for index 2 current graphs, they are labeled $[0]$ and $[1]$. Given a circuit, its \emph{log} replaces each arc $e^\pm$ along the face-boundary walk of the circuit with the corresponding current $\alpha(e^\pm)$ if the walk is in its original orientation and $-\alpha(e^\pm)$ when the orientation is reversed.

The \emph{excess} of a vertex is the sum of the currents on the arcs leaving the vertex. We say that a vertex satisfies \emph{Kirchhoff's current law (KCL)} if its excess is 0 in the current group. Our goal is to find triangular or near-triangular embeddings of circulant graphs $C(n, S)$, so we restrict ourselves to current graphs that satisfy the following standard properties:

\begin{enumerate}
\item[(C1)] Each vertex is of degree 1, 2, or 3.
\item[(C2)] The log of each circuit contains each element of $\{\pm s\mid s \in S\}$ exactly once.
\item[(C3)] KCL holds at every vertex of degree 3. 
\item[(C5)] If circuit $[a]$ traverses arc $e^+$ and circuit $[b]$ traverses arc $e^-$, then $b-a \equiv \alpha(e^+)\pmod{k}$, where $k$ is the index of the current graph. 
\end{enumerate}

All vertices that have degree less than 3 must fall under one of the following categories:
\begin{enumerate}
\item[(T1)] A vertex of degree 1 whose excess is of order 2 in $\mathbb{Z}_{2m}$,
\item[(T2)] A vertex of degree 1 whose excess is of order 3 in $\mathbb{Z}_{2m}$,
\item[(T3)] A vertex of degree 1 whose excess generates the subgroup of even elements, or
\item[(T4)] A vertex of degree 2 whose incident arcs have odd currents, and whose excess generates the subgroup of even elements.
\end{enumerate}

Since such vertices violate KCL, we call them \emph{vortices}. Vortices of type (T1) and (T2) generate faces of length 2 and 3, respectively. The length 2 ``lunes'' can be suppressed by deleting one of the two edges, so a vortex of type (T1) does not count as a violation of property (C2). In fact, when the index of the current graph is 1, the order 2 current must be incident with a vortex of type (T1). We follow the traditional convention of having the current appear once in the log, and drawing the edge without a direction and without its degree 1 endpoint.  

Vortices of the other two types (T3) and (T4) generate a small number of long faces. In our current graphs, we label them with as many letters as the number of faces they generate. Vortices of type (T3) generate $m$-sided faces, and vortices of type (T4) generate $2m$-sided faces. If the index is 1, then two such faces are generated per vortex, and if the index is 2, then only one such face is generated. The remaining vertices are of degree 3 and satisfy KCL. Hence, they generate triangular faces. 

We follow the convention where vortex labels are added to the log as they are encountered, which serves as a reminder as to where the nontriangular faces will be. Current graphs generate a \emph{derived embedding} where the vertices are the elements of the current group, and when the above properties are satisfied, the derived embedding is of the circulant graph $C(n, S)$ with triangular faces except at faces generated by vortices of type (T3) and (T4). 

The rotation system of the derived embedding is specified in the following way: to generate the rotation at vertex $i \in \mathbb{Z}_{2m}$, temporarily ignore any vortex labels and add $i$ to each numbered element of the log of $[i \bmod k]$, where $k$ is the index. The signature of an edge is $+1$ if its corresponding edge in the current graph is bidirectional, and $-1$, otherwise. Then, subdivide each nontriangular face corresponding to a vortex of type (T3) or (T4) with a vertex with the same label. When the above properties are satisfied, the derived embedding, with its subdivision vertices, is triangular. 

The index 2 current graph in Figure~\ref{fig-k13} satisfies the above properties and has two vortices of type (T3). Its logs are:

$$\begin{array}{rrrrrrrrrrrrrrrrr}
\lbrack0\rbrack. & 10 & x_0 & 2 & 9 & 3 & 1 & 5 & 6 & 11 & 4 & 8 & 7 \\
\lbrack1\rbrack. & 2 & x_1 & 10 & 9 & 6 & 3 & 5 & 1 & 7 & 8 & 4 & 11 \\
\end{array}$$

\begin{figure}[!ht]
\centering
\includegraphics[scale=0.85]{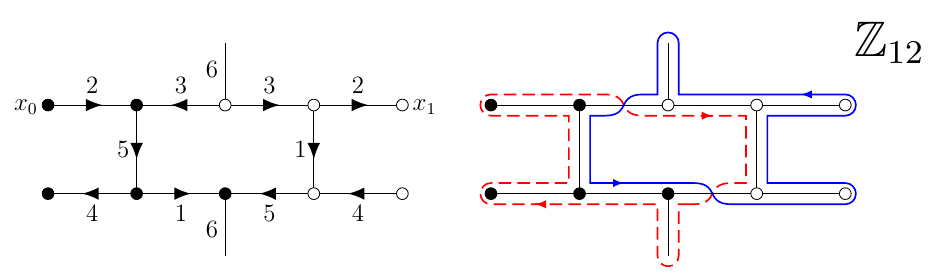}
\caption{A current graph and its face boundaries.}
\label{fig-k13}
\end{figure}

Following the procedure above, the final rotation system is:

$${\small\begin{array}{rrrrrrrrrrrrrrrrr}
0. & 10 & x_0 & 2 & 9 & 3 & 1 & 5 & 6 & 11 & 4 & 8 & 7 \\
1. & 3 & x_1 & 11 & 10 & 7 & 4 & 6 & 2 & 8 & 9 & 5 & 0 \\
2. & 0 & x_0 & 4 & 11 & 5 & 3 & 7 & 8 & 1 & 6 & 10 & 9 \\
3. & 5 & x_1 & 1 & 0 & 9 & 6 & 8 & 4 & 10 & 11 & 7 & 2 \\
4. & 2 & x_0 & 6 & 1 & 7 & 5 & 9 & 10 & 3 & 8 & 0 & 11 \\
5. & 7 & x_1 & 3 & 2 & 11 & 8 & 10 & 6 & 0 & 1 & 9 & 4 \\
6. & 4 & x_0 & 8 & 3 & 9 & 7 & 11 & 0 & 5 & 10 & 2 & 1 \\
7. & 9 & x_1 & 5 & 4 & 1 & 10 & 0 & 8 & 2 & 3 & 11 & 6 \\
8. & 6 & x_0 & 10 & 5 & 11 & 9 & 1 & 2 & 7 & 0 & 4 & 3 \\
9. & 11 & x_1 & 7 & 6 & 3 & 0 & 2 & 10 & 4 & 5 & 1 & 8 \\
10. & 8 & x_0 & 0 & 7 & 1 & 11 & 3 & 4 & 9 & 2 & 6 & 5 \\
11. & 1 & x_1 & 9 & 8 & 5 & 2 & 4 & 0 & 6 & 7 & 3 & 10 \\
x_0. & 0 & 10 & 8 & 6 & 4 & 2 \\
x_1. & 1 & 3 & 5 & 7 & 9 & 11 \\
\end{array}}$$

This rotation system is a triangular embedding of the balanced split-complete graph $B_{13}$. Merging the vertices $x_0$ and $x_1$ with a handle results in a genus embedding of $K_{13}$. 

When the current graph is embedded in an orientable surface, the derived embedding will also be orientable. In one example (Section~\ref{sec-case8non}), we describe nonorientable current graphs of index 2 whose derived embeddings are still nonorientable. However, it is possible for a nonorientable current graph to have an orientable derived embedding. We say that a nonorientable index 1 current graph is an \emph{orientable cascade} if it satisfies an additional property:
\begin{enumerate}
\item[(O*)] An edge is bidirectional if and only if its current is even.
\end{enumerate}
If (O*) is satisfied, then an edge is twisted if and only if it is incident with one even vertex and one odd vertex. Thus, the derived embedding is orientable, since flipping all of the odd vertices results in a pure rotation system. 

As an example, consider the orientable cascade in Figure~\ref{fig-o28}. Its log is:
$$\begin{array}{rrrrrrrrrrrrrrrrrrrrrrrrrrr}
\lbrack0\rbrack. & 1 & 10 & 18 & 20 & 13 & 8 & 14 & 15 & 17 & 4 & 22 & 7 & 11 & \dots \\ & 19 & y & 21 & z & 23 & 9 & 2 & 6 & 16 & 5 & w & 3 & x \\
\end{array}$$
The derived embedding is of the graph $O_{24}+\overline{K_4}$, where $G+H$ denotes the graph join of $G$ and $H$. Since the current graph satisfies (O*), the derived embedding is orientable. Later on, we show how it can be converted into a genus embedding of the octahedral graph $O_{28}$.

\begin{figure}[!ht]
\centering
\includegraphics[scale=0.85]{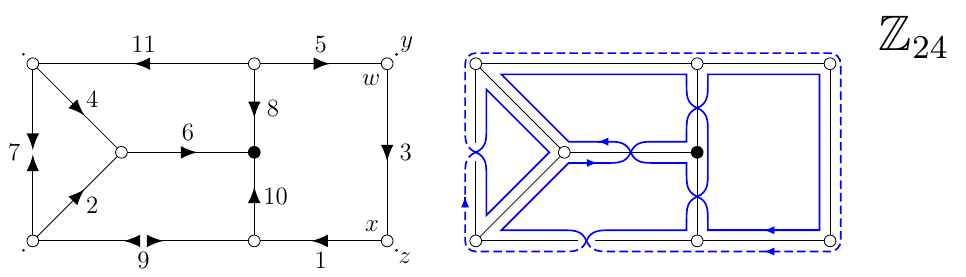}
\caption{An orientable cascade with two vortices of type (T4).}
\label{fig-o28}
\end{figure}

\section{Jungerman ladders}\label{sec-jungerman}

Finding infinite families of current graphs is difficult, so typically one restricts the ``shape'' of the underlying graphs. A standard trick is to include a \emph{ladder} of varying length, such as the ones in Figure~\ref{fig-exsnake}, where there are four partial circuits that traverse the length of the ladder. The vertical arcs are referred to as \emph{rungs} and the horizontal arcs are divided into \emph{top} and \emph{bottom horizontals}. The currents assigned to these ladders are often derived from some kind of labeling problem, e.g., graceful labelings of paths~\cite{Goddyn-Exponential}, where the rungs correspond to differences between consecutive labels. 

\begin{figure}[!ht]
\centering
    \begin{subfigure}[b]{0.98\textwidth}
    \centering
        \includegraphics[scale=0.85]{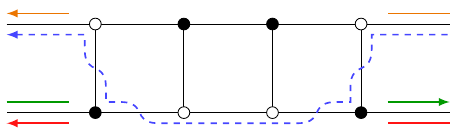}
        \caption{}
        \label{subfig-partial}
    \end{subfigure}
    \begin{subfigure}[b]{0.98\textwidth}
    \centering
        \includegraphics[scale=0.85]{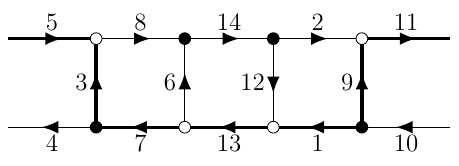}
        \caption{}
        \label{subfig-snake}
    \end{subfigure}
\caption{An example of a small Jungerman ladder highlighting the path of odd currents.}
\label{fig-exsnake}
\end{figure}

Property (O*) for orientable cascades can be accounted for in a ladder by making exactly one of the two leftmost horizontals odd. If KCL is satisfied on every vertex of the ladder, regardless of the labels on the rungs, there will be a path of odd currents that run through the entire length of the ladder. Suppose that exactly one of the partial circuits is in ``reversed'' behavior. There is a unique assignment of the rotations such that the reversed partial circuit traverses exactly that path of odd currents. Thus, property (O*) is satisfied on each of the ladder's edges. We call such a construction a \emph{Jungerman ladder}, an example of which is given in Figure~\ref{fig-exsnake}(b).

In all of our constructions, we use the shorthand in Figure~\ref{fig-exladder}. The orientations on the rungs alternate and the labels on the rungs form an arithmetic sequence from left to right with step size equal to the number in the box. The horizontals are filled in by enforcing KCL at each vertex. We always use an odd number as the step size, so the rotations on the vertices will be one of two possible ``checkerboard'' patterns. The smallest integer $s$ for which such a ladder is defined, as seen in the bottom-left example in Figure~\ref{fig-exladder}, is the one which causes the left- and right-hand sides to coincide, creating a ladder with 0 rungs. 

\begin{figure}[!ht]
\centering
\includegraphics[scale=0.85]{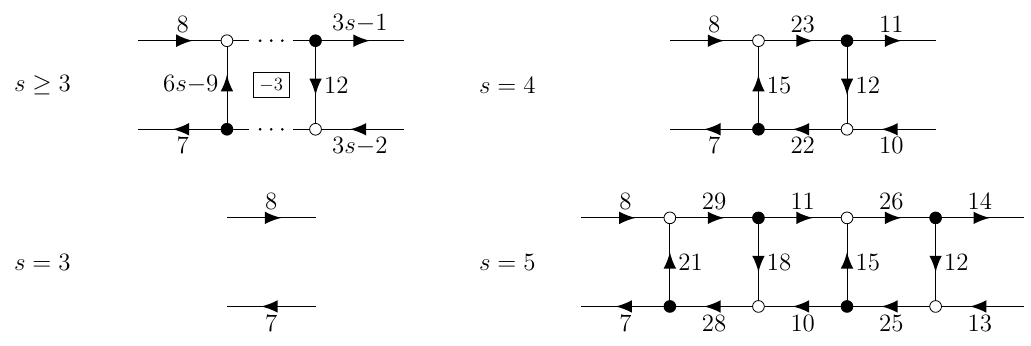}
\caption{A typical example of a family of Jungerman ladders.}
\label{fig-exladder}
\end{figure}

An example of an infinite family of orientable cascades using a Jungerman ladder is shown in Figure~\ref{fig-case4cascade}. These current graphs generate orientable triangular embeddings of $K_{12s+4}$ for even values of $s$, even when $s = 0$. For odd values of $s$, the resulting current graph has the wrong index, and we later show that this cannot be circumvented. 

\begin{figure}[!ht]
\centering
\includegraphics[scale=0.85]{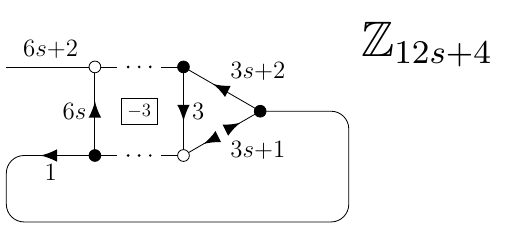}
\caption{A simple family of orientable cascades.}
\label{fig-case4cascade}
\end{figure}

\section{The good news: octahedral graphs}

Jungerman and Ringel~\cite{JungermanRingel-Octa} constructed triangular embeddings of $O_{2m}$ when $m \equiv 0,1 \pmod{3}$ using orientable cascades. For the remaining octahedral graphs, they conjectured that the genus matches the Euler lower bound, as well. We resolve the missing cases, also using orientable cascades:

\begin{theorem}
$$\gamma(O_{2m}) = \left\lceil \frac{(m-1)(m-3)}{3} \right\rceil$$
for $m \equiv 2 \pmod{3}$.
\end{theorem}
\begin{proof}
For sufficiently large $m$, we start with a triangular embedding of the graph join $O_{2m-4}+\overline{K_4}$ and supply the missing edges using one handle. The vertices of $O_{2m-4}$ are numbered $0, 1, \dotsc, 2m-5$ and the vertices of $\overline{K_4}$ are $w, x, y, z$. As the naming scheme suggests, we will use orientable cascades with current group $\mathbb{Z}_{2m-4}$ whenever possible. 

This join is disconnected when $2m = 4$, and Ellingham and Schroeder~\cite{EllinghamSchroeder-Ham2} showed that when $2m = 10$, no orientable triangular embedding exists. We find genus embeddings in other ways: $O_4$ is just the planar 4-cycle $C_4$, and a genus embedding of $O_{10} \cong K_{10}-5K_2$ can be found by deleting two additional independent edges from the triangular embedding of $K_{10}-3K_2$ in Appendix~\ref{app-sporadic}. The construction of the latter embedding is due to Jonathan L.\ Gross (personal communication).

Returning to the general case, we use the two \emph{ad hoc} rotation systems for $2m = 16, 22$ listed in Appendix \ref{app-sporadic} and the orientable cascades in Figures \ref{fig-o28}, \ref{fig-smallcascades}--\ref{fig-case4-cascades-odd}. Transforming these triangular embeddings of $O_{2m-4}+\overline{K_4}$ into genus embeddings of $O_{2m}$ will proceed in a relatively unified fashion. The rotation at vertex $0$ is always of the form:
$${\begin{array}{rrrrrrrrrrrrrrrrrr}
0. & \dots & w & 3 & x & \dots & y & 2m{-}7 & z & \dots  \\
\end{array}}$$
We first apply a few modifications to the derived embeddings using \emph{edge flips}, where an edge is deleted and the other diagonal of the resulting quadrilateral face is added in its place. Delete the edges $(0,3)$ and $(0,2m{-}7)$, and add $(w,x)$ and $(y,z)$ in the newly created quadrilaterals, as shown in Figure~\ref{fig-octadj}(a). For each embedding, there is a sequence of edge flips, shown in Figure \ref{fig-edgeflip}, that starts by deleting the edge $(2m{-}7, x)$ and ends with adding back the edge $(0,3)$.

\begin{figure}[!ht]
    \centering
    \begin{subfigure}[b]{0.49\textwidth}
    \centering
        \includegraphics[scale=0.7]{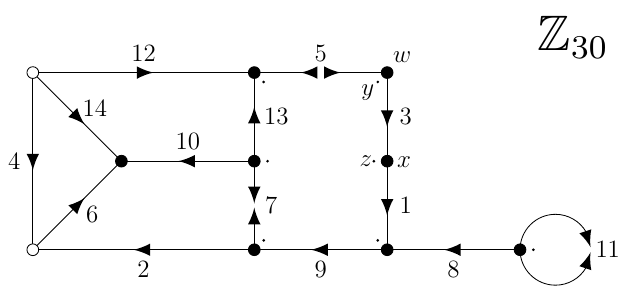}
        \caption{}
        \label{subfig-o34}
    \end{subfigure}
    \begin{subfigure}[b]{0.49\textwidth}
    \centering
        \includegraphics[scale=0.7]{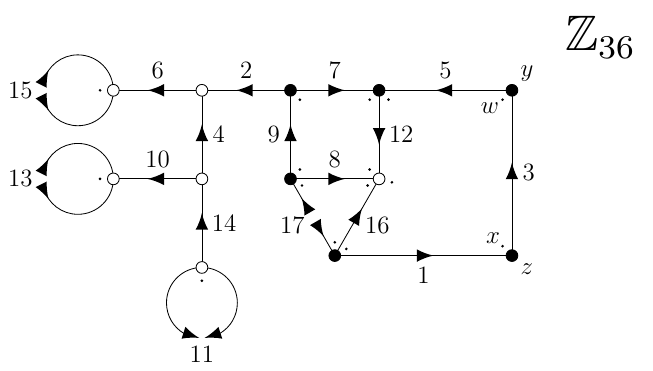}
        \caption{}
        \label{subfig-o40}
    \end{subfigure}
\caption{Orientable cascades for $O_{30}+\overline{K_4}$ (a) and $O_{36}+\overline{K_4}$ (b).}
\label{fig-smallcascades}
\end{figure}

\begin{figure}[!ht]
\centering
\includegraphics[scale=0.8]{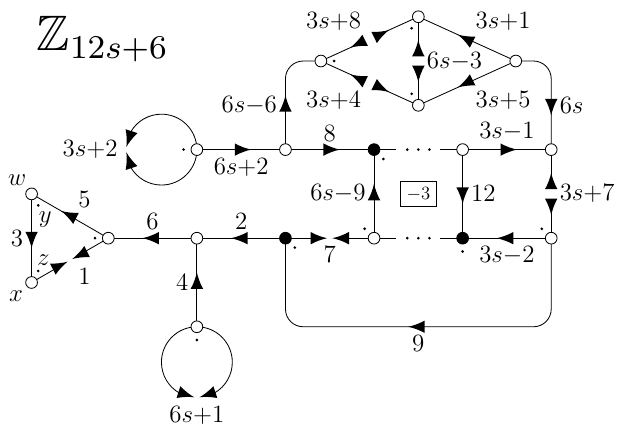}
\caption{$O_{12s+6}+\overline{K_4}$, odd $s \geq 3$.}
\label{fig-case10-cascades-odd}
\end{figure} 

\begin{figure}[!ht]
\centering
\includegraphics[scale=0.8]{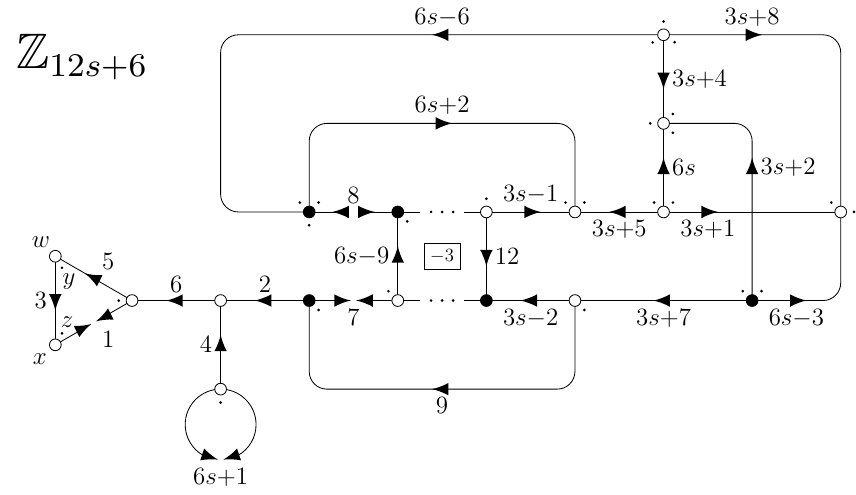}
\caption{$O_{12s+6}+\overline{K_4}$, even $s \geq 4$.}
\label{fig-case10-cascades-even}
\end{figure} 

\begin{figure}[!ht]
\centering
\includegraphics[scale=0.8]{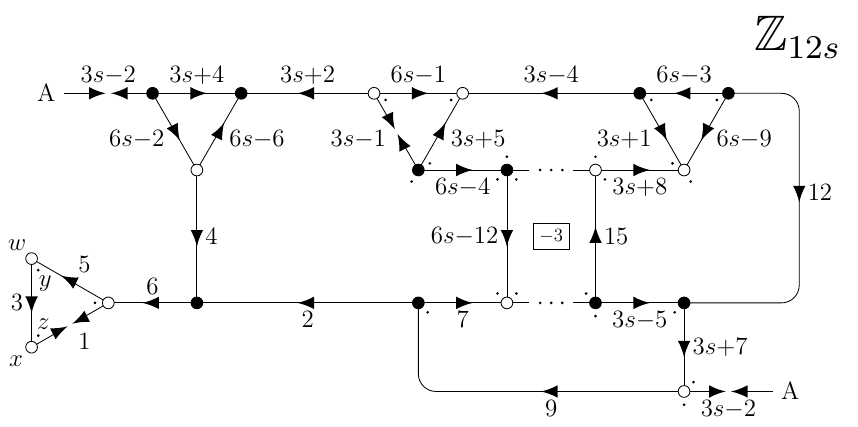}
\caption{$O_{12s}+\overline{K_4}$, even $s \geq 4$.}
\label{fig-case4-cascades-even}
\end{figure} 

\begin{figure}[!ht]
\centering
\includegraphics[scale=0.8]{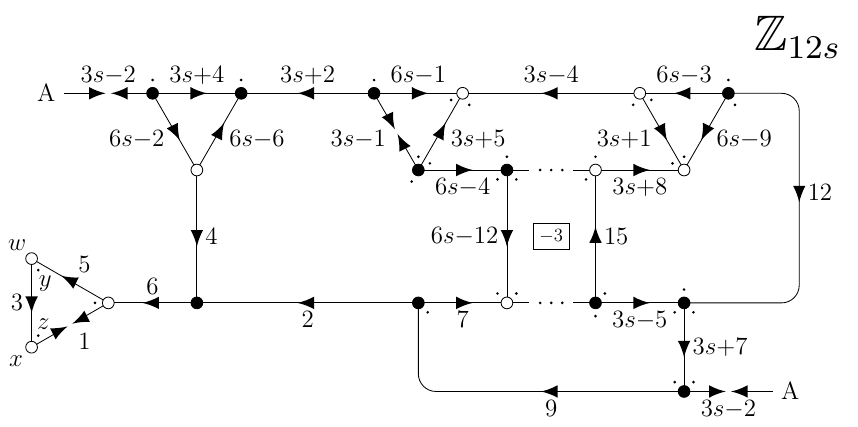}
\caption{$O_{12s}+\overline{K_4}$, odd $s \geq 5$.}
\label{fig-case4-cascades-odd}
\end{figure} 

\begin{figure}[!ht]
    \centering
    \begin{subfigure}[b]{0.65\textwidth}
    \centering
        \includegraphics[scale=0.8]{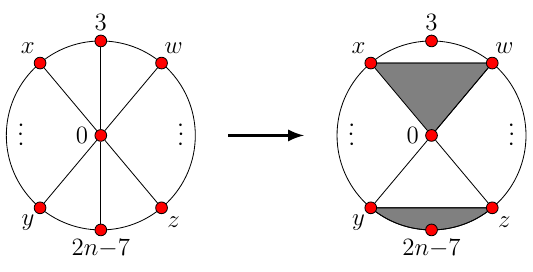}
        \caption{}
        \label{subfig-octadd1}
    \end{subfigure}
    \begin{subfigure}[b]{0.34\textwidth}
    \centering
        \includegraphics[scale=0.8]{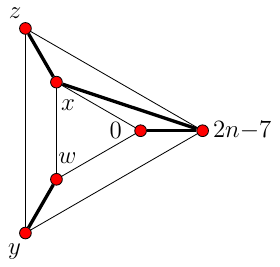}
        \caption{}
        \label{subfig-octadd2}
    \end{subfigure}
\caption{Steps to adding four edges with one handle.}
\label{fig-octadj}
\end{figure}

\begin{figure}[!ht]
\centering
\includegraphics[scale=0.8]{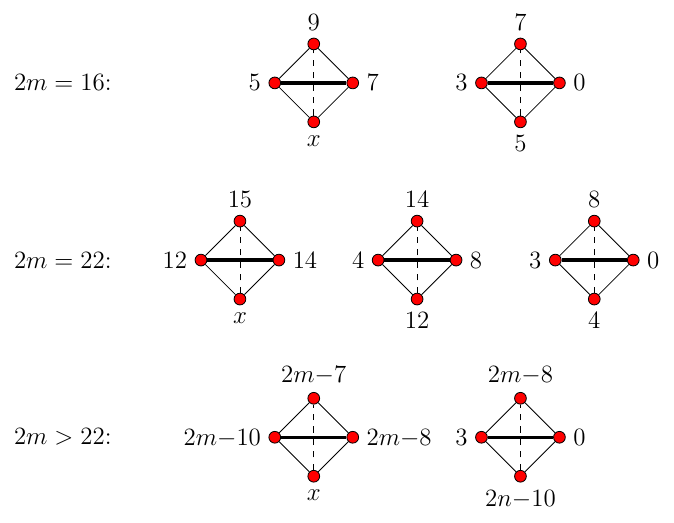}
\caption{Sequences of edge flips that regain the edge $(0,3)$.}
\label{fig-edgeflip}
\end{figure}

Finally, Figure~\ref{fig-octadj}(b) shows how one can insert a handle between the shaded faces in Figure~\ref{fig-octadj}(a) to add the new edges $(x,z)$ and $(w,y)$, and reinsert the deleted edges $(0, 2m{-}7)$ and $(x, 2m{-}7)$. The final embedding is a minimum genus embedding of the octahedral graph $O_{2m}$, which can be checked by noting that there are two quadrilateral faces formed by the handle, and all the other faces are triangular.

\end{proof}

\section{The bad news: nonexistence of orientable cascades}
\label{sec-nonexist}

Jungerman \cite{Jungerman-OrientableCascades} showed that orientable cascades can only exist when a certain parity condition is satisfied. The original proof is somewhat complicated and relies on the fact that the derived embedding is of a complete graph. We prove a generalization to other circulant graphs using a simpler argument that interprets orientable cascades as one-face embeddings in a nonorientable surface, as opposed to purely combinatorial objects. In this section, we assume that we are working with an orientable cascade whose derived embedding is of a circulant graph $C(n, S)$, where the order 2 element $n/2$ is in the generating set $S$. The number of vertices, edges, and twisted edges are denoted by $|V|$, $|E|$, and $\varepsilon$, respectively. We also write $t_i$ for the number of vortices of type (T$i$), for $i = 1, 2, 3, 4$. Since the order 2 element is present, $t_1 = 1$. 

Our proof of the parity condition proceeds in three steps. First, we show that the order of the current group must be a multiple of 4, allowing us to perform calculations modulo 4. Next, we prove two separate parity conditions on the number of twisted edges. Finally, we show that these two parity conditions cannot always be satisfied at the same time. 

\begin{proposition}
If there exists an orientable cascade that generates a derived embedding of $C(n, S)$, then $n$ is a multiple of $4$. 
\label{prop-mult4}
\end{proposition}
\begin{proof}
The edge incident with the vortex of type (T1) must be bidirectional, and hence its current, $n/2$, is even. 
\end{proof}

Bernardi and Chapuy \cite{BernardiChapuy} introduced the idea of a \emph{canonical} rotation system for embeddings of graphs where every vertex has degree 1 or 3: for each vertex, choose the orientation (via a vertex flip, if necessary) that results in fewer reversed corners at that vertex. They showed that, in the canonical rotation system, every twisted edge is unidirectional. We adapt this result for our situation where some vertices may have degree 2. We allow either orientation for each vertex of degree 2, leading to potentially multiple rotation systems that can all be called canonical. 

\begin{lemma}
In a canonical rotation system of an orientable cascade, all twisted edges are unidirectional, or equivalently, have odd current.
\label{lem-canonical}
\end{lemma}
\begin{proof}
If the orientable cascade has any vortices of type (T4), temporarily smooth away those degree 2 vertices. Lemma 2 of Bernardi and Chapuy \cite{BernardiChapuy} shows that in the canonical rotation system, every twisted edge is unidirectional. Now, we re-introduce those degree 2 vertices via subdivision, giving them an arbitrary orientation. Some of the subdivision edges may be twisted, but since they are incident with a vortex of type (T4), they must have odd current, and hence be unidirectional.
\end{proof}

The parity condition arises specifically from the number of twisted edges in a canonical rotation system. This number is constrained simultaneously by both the Euler polyhedral formula and the vortices. 

\begin{lemma}
In a canonical rotation system of an orientable cascade, $|V|+|E|$ and the number of twisted edges $\varepsilon$ have opposite parity. 
\label{lem-opp}
\end{lemma}
\begin{proof}
By Lemma \ref{lem-canonical}, the twisted edges are all unidirectional. Switching the signature of any such edge either merges two faces together, or splits a face into two faces. In other words, the number of faces changes parity for each switch, so if we switch each of the $\varepsilon$ twisted edges, the number of faces $|F|$, which started out at 1, is of the opposite parity from $\varepsilon$ in the final embedding. The final embedding is in an orientable surface, which has even Euler characteristic. Thus, by the Euler polyhedral formula, $|F|$ has the same parity as $|V|+|E|$.
\end{proof}

The relationship between vortices and twisted edges is a consequence of a ``global'' KCL:

\begin{proposition}[e.g., Ringel\ {\cite[p.31]{Ringel-MapColor}}]
Let $G = (V, E)$ be a graph and $\alpha\colon E(G)^+ \to \Gamma$ be an arc-labeling satisfying $\alpha(e^+) = -\alpha(e^-)$ for every edge $e \in E$. Then the sum of the excesses of all the vertices is 0.
\label{prop-gkcl}
\end{proposition}
\begin{proof}
Each edge $e$ contributes $\alpha(e^+)$ to the excess of one of its endpoints, and $-\alpha(e^+)$ to that of the other.
\end{proof}

\begin{lemma}
In a canonical rotation system of an orientable cascade with current group $\mathbb{Z}_{4k}$, $\varepsilon +t_3+t_4 \equiv k \pmod{2}$.
\label{lem-vortexsum}
\end{lemma}
\begin{proof}
Because of twisted edges, the arc-labeling of an orientable cascade does not initially satisfy the conditions needed for Proposition \ref{prop-gkcl}. Let $(u,v)$ be a twisted edge where both of its arcs have an odd current of $\gamma$. Following Figure \ref{fig-subdivide}, we subdivide the edge by adding a new vertex $w$ and give the label $\gamma$ to the arcs $u{\to}w$ and $v{\to}w$, and $-\gamma$ to $w{ \to}u$ and $w{\to}v$. Repeating this process for every twisted edge results in an arc-labeled graph satisfying Proposition \ref{prop-gkcl}, the global KCL, at the cost of new vertices with nonzero excess.

Since the order of the current group is a multiple of 4, the global KCL is also satisfied modulo 4. The excess of every subdivision vertex is $2 \pmod{4}$ because it is twice some odd number. The excess of every vortex of type (T3) and (T4) is $2 \pmod{4}$ since they have to generate the subgroup of even elements. A vortex of type (T2) must have excess $0 \pmod{4}$, because if $3$ divides $4k$, it must divide $k$. The excess of a vortex of type (T1) depends on the current group itself, so at this point, we leave it as $2k$. Since all other vertices satisfy KCL, Proposition \ref{prop-gkcl} tells us that the excesses of the vortices and subdivision vertices sum to 0. Thus,
$$2\varepsilon + 2t_3 + 2t_4 + 2k \equiv 0 \pmod{4},$$
or equivalently,
$$\varepsilon + t_3 + t_4 \equiv k \pmod{2}.$$
\end{proof}

A similar parity condition to Lemma \ref{lem-vortexsum} specifically for complete graphs was proven by Jungerman \cite[Thm.\ 4]{Jungerman-OrientableCascades}.

\begin{figure}[!ht]
\centering
\includegraphics[scale=0.8]{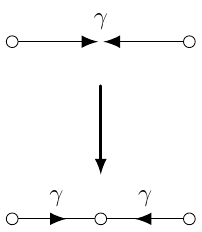}
\caption{Removing twisted edges from an orientable cascade.}
\label{fig-subdivide}
\end{figure}

With these technical ingredients in place, we can finally prove our strengthening of Jungerman's impossibility result. Jungerman's original exposition stops short of giving a general statement, instead applying Lemma \ref{lem-vortexsum} directly to each of the families of graphs he was considering. Further generalizations of the forthcoming result can be made for other circulant graphs that are not near-complete, such as the symmetric complete quadripartite graphs considered in Jungerman \cite{Jungerman-Quadripartite}. 

\begin{theorem}
Let $G_s = C(12s+4j, S_s)$ be a family of circulant graphs, for some fixed $j$ and $s = 1, 2, \dotsc$, where the order 2 element $6s+2j$ is in each generating set $S_s$, and that $|\{1, \dotsc, 6s+2j\} \setminus S_s|$ is constant for all $s$ (i.e., the ``same'' elements are missing from the generating set). Fix the number $t_i$ of each type of vortex. For at least one parity of $s$, there do not exist orientable cascades with $t_1$, $t_2$, $t_3$, and $t_4$ vortices of type (T1), (T2), (T3), and (T4), respectively, whose derived embeddings are $G_s$.
\label{thm-nonexistence}
\end{theorem}
\begin{proof}
It suffices to show that, for all $s'$, such orientable cascades cannot exist simultaneously for both $s'$ and $s'+1$. Suppose both exist. Since $G_{s'+1}$ has 12 more vertices than $G_{s'}$, its generating set also has 12 more elements. Thus, the orientable cascade for $G_{s'+1}$ has 6 more edges than that of $G_{s'}$. Fixing the number of each type of vortex forces there to be 4 more vertices of degree 3. By Lemma \ref{lem-opp}, the parity of the number of twisted edges in their canonical rotation systems is the same, so the left-hand side of Lemma \ref{lem-vortexsum} is the same for both $s'$ and $s'+1$, as well. However, the parity of the right-hand side, $k = 3s+j$, is different between $s'$ and $s'+1$, which is a contradiction.   
\end{proof}

Gross and Tucker \cite[p.231]{GrossTucker} considered whether there exist orientable cascades that generate triangular orientable embeddings of $K_{12s+4}$, for all $s \geq 0$. Figure \ref{fig-case4cascade} solves the even $s$ case, but Theorem \ref{thm-nonexistence} shows that they do not exist for odd $s$. A similar parity constraint will be true for all of the forthcoming families of graphs. However, all is not lost, as we will still be able to solve the other parity by other means.

\section{Some good news again: families of index 2 current graphs}

Recall that the derived embedding of an orientable cascade is oriented by flipping all of the odd-numbered vertices. One can alternatively interpret this rotation system as having been generated from an \emph{orientable} index 2 current graph, i.e., every orientable cascade has an ``orientable double cover.''

While orientable cascades and their equivalent symmetric index 2 cousins may not always exist, one can at least salvage the Jungerman ladder as a useful construction technique: the double cover of a Jungerman ladder consists of two ladders with reversed arc directions and vertex rotations. As an example, the orientable cascades in Figure \ref{fig-case4cascade} can be interpreted as the index 2 current graphs in Figure \ref{fig-case4double}(a). The family in Figure \ref{fig-case4double}(b), which generate triangular embeddings of $K_{12s+4}$, for odd $s$, reuses much of the ladders, but the rest of the graph crucially introduces some asymmetry. The two families combined cover all values of $s$, solving a problem of Ringel \cite[p.149]{Ringel-MapColor}.

\begin{figure}[!ht]
\centering
    \begin{subfigure}[b]{0.4\textwidth}
    \centering
        \includegraphics[scale=0.8]{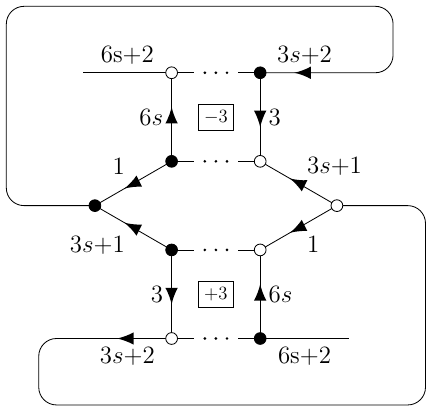}
        \caption{}
        \label{subfig-case4even}
    \end{subfigure}
    \begin{subfigure}[b]{0.59\textwidth}
    \centering
    \includegraphics[scale=0.8]{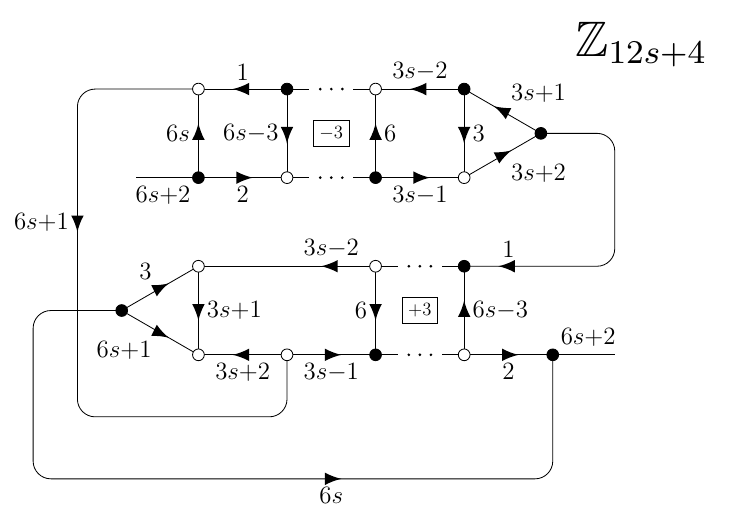}
        \caption{}
        \label{subfig-case4odd}
    \end{subfigure}
\caption{Index 2 current graphs for $K_{12s+4}$, for $s$ even (a) and $s$ odd (b).}
\label{fig-case4double}
\end{figure}

For each family of graphs considered in this section, we present two families of index 2 current graphs, though in some cases, the symmetry in one of the two families indicates that it is equivalent to a family of orientable cascades.

\subsection{Hamiltonian cycle complements}

An open problem of White~\cite{ArchdeaconWhite} asks for the genus of the complete graph minus a Hamiltonian cycle, which one might conjecture to be equal to its Euler lower bound
$$\gamma(K_n-C_n) \geq \left\lceil \frac{n^2-9n+12}{12} \right\rceil$$
except at a few small values of $n$. A triangular embedding is only possible when $n \equiv 0, 9 \pmod{12}$, and these embeddings were constructed for the $n \equiv 9\pmod{12}$ case by Gross and Alpert~\cite{GrossAlpert}. We find triangular embeddings of $K_{12s}-C_{12s}$, the other feasible residue, using the current graphs in Figures~\ref{fig-h12-even} and \ref{fig-h12-odd}. The currents $1$ and $-1$ are missing from both circuits, so in the derived graph, the absent edges form the missing Hamiltonian cycle. These current graphs are defined for $s \geq 2$, so the $s = 1$ case is treated in Appendix~\ref{app-sporadic}.

\begin{figure}[!ht]
\centering
\includegraphics[scale=0.8]{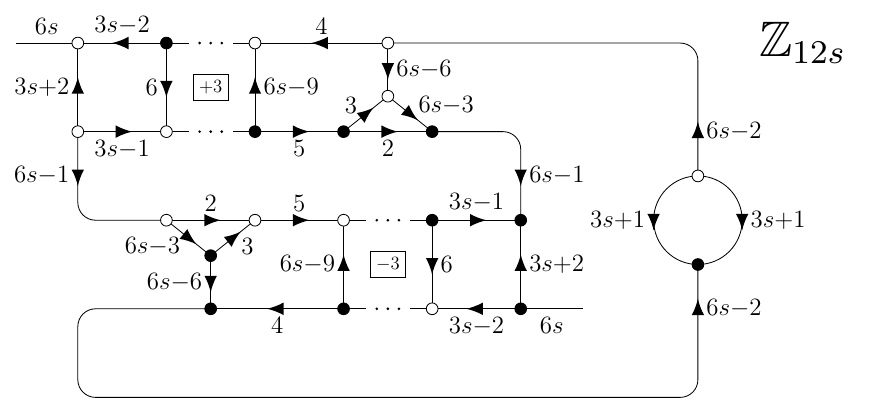}
\caption{$K_{12s}-C_{12s}$, $s$ even.}
\label{fig-h12-even}
\end{figure}

\begin{figure}[!ht]
\centering
\includegraphics[scale=0.8]{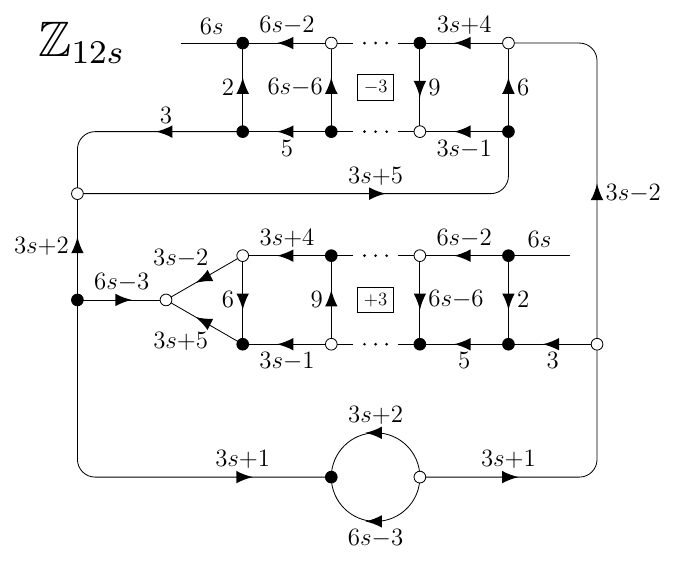}
\caption{$K_{12s}-C_{12s}$, $s$ odd, $s \geq 3$.}
\label{fig-h12-odd}
\end{figure}

\subsection{$K_{12s+8}-K_2$, non-orientable}\label{sec-case8non}

The original proof of the Map Color Theorem for nonorientable surfaces using current graphs~\cite[p.142]{Ringel-MapColor} tackled the $n \equiv 8 \pmod{12}$ case with a simple family of current graphs and an elegant surgery operation involving two crosscaps. An alternate route starts with a triangular embedding of $K_{12s+8}-K_2$ and adds the missing edge with a single crosscap. Ringel~\cite{Ringel-K8} proved that no such embedding exists for $s = 0$. 

Korzhik~\cite{Korzhik-Case8} gave a complex index 2 construction for $s \geq 1$, which we simplify using Jungerman ladders. Any triangular embedding of $K_{12s+8}-K_2$ is in a surface of odd Euler characteristic, so it must be nonorientable. Hence, any current graph construction must have twisted edges. Korzhik sidestepped this additional difficulty by using only one twisted edge, incident with the same vertex. Thus, the face boundaries can be oriented so that only the self-loop is unidirectional. Our constructions utilize this same trick, as seen in Figures~\ref{fig-c8-even} and \ref{fig-c8-odd}. We note that $3s+1$ (the current on the self-loop and the excess of the vortices in Figure \ref{fig-c8-odd}) is even when $s$ is odd.

While these triangular embeddings are not the easiest path to determining the nonorientable genus of $K_{12s+8}$, they are minimal triangulations (see Jungerman and Ringel~\cite{JungermanRingel-Minimal}) of the surfaces they embed in. 

\begin{figure}[!ht]
\centering
\includegraphics[scale=0.8]{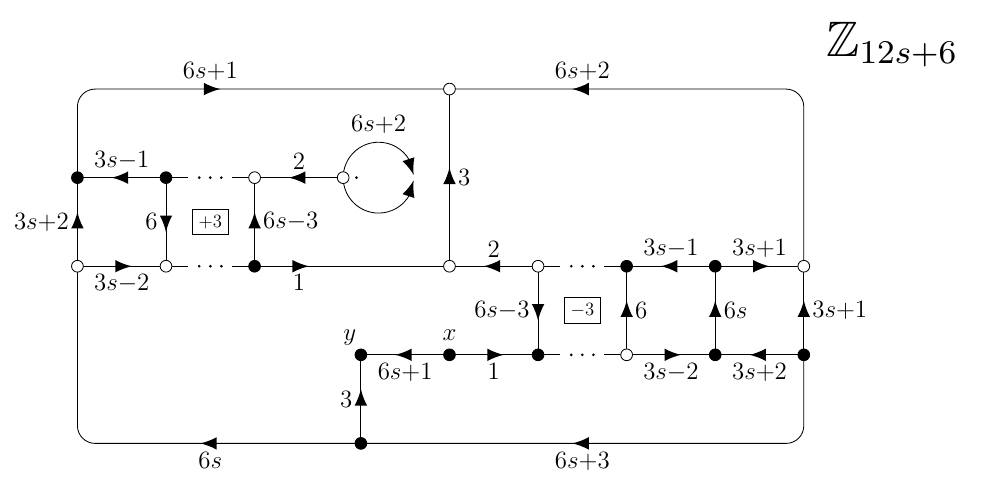}
\caption{$K_{12s+8}-K_2$, even $s \geq 2$.}
\label{fig-c8-even}
\end{figure}

\begin{figure}[!ht]
\centering
\includegraphics[scale=0.8]{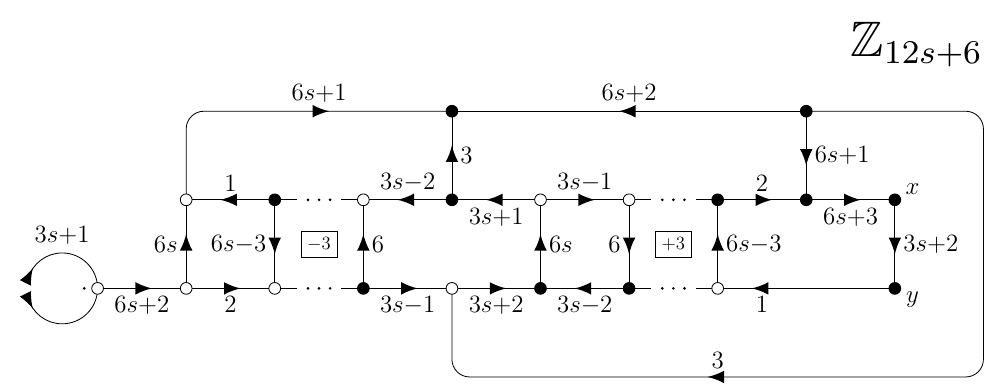}
\caption{$K_{12s+8}-K_2$, odd $s \geq 1$.}
\label{fig-c8-odd}
\end{figure}

\subsection{$K_{12s+2}-K_2$}

Jungerman~\cite{Jungerman-KnK2} gave the first complete index 2 solution for the $n \equiv 2 \pmod{12}$ case of the Map Color Theorem using his namesake ladders. Two simpler families, as seen in Figures \ref{fig-c2-odd} and \ref{fig-c2-even}, can be drawn without crossings. Despite discovering the current graphs in Figure \ref{fig-c2-odd}, Jungerman did not publish it, perhaps in favor of a more unified solution.

\begin{figure}[!ht]
\centering
\includegraphics[scale=0.8]{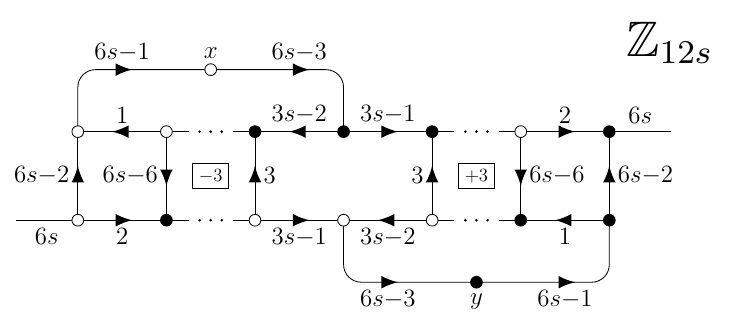}
\caption{$K_{12s+2}-K_2$, odd $s \geq 1$.}
\label{fig-c2-odd}
\end{figure}

\begin{figure}[!ht]
\centering
\includegraphics[scale=0.8]{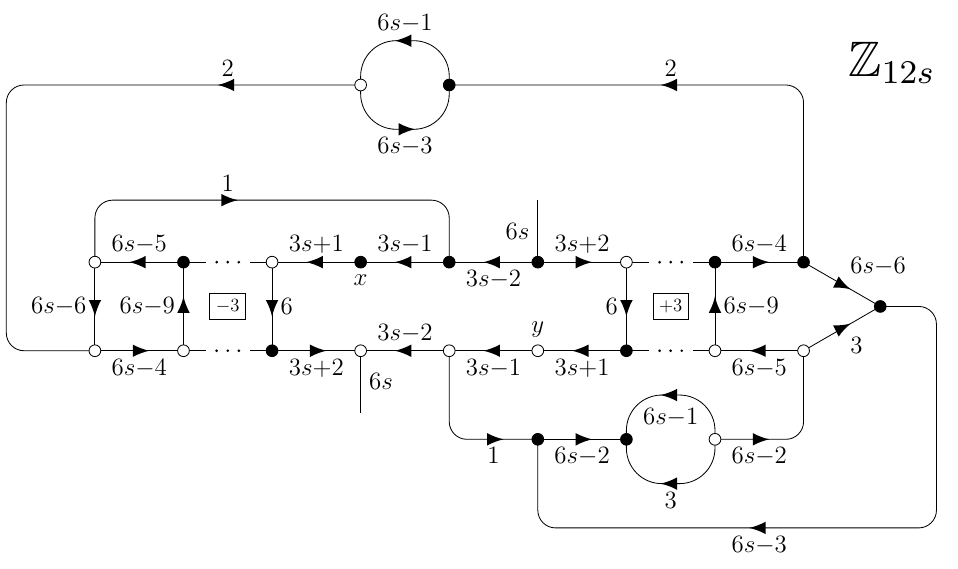}
\caption{$K_{12s+2}-K_2$, even $s \geq 2$.}
\label{fig-c2-even}
\end{figure}

\subsection{$B_{12s+1}$}

We generalize the current graph in Figure \ref{fig-k13} to all balanced split-complete graphs of the form $B_{12s+1}$ for all $s \geq 1$ using the two families in Figures \ref{fig-sc1-even} and \ref{fig-sc1-odd}. As a side note, the resulting minimum genus embeddings of $K_{12s+1}$ form an alternate proof of part of the main result of Sun~\cite{Sun-FaceDist}, since they each have one nontriangular face. Unlike the previous constructions, the increment in the rungs of the Jungerman ladders is 1 to accommodate for the vortex of type (T2): the order 3 element $4s \in \mathbb{Z}_{12s}$ would normally be inside of the Jungerman ladder if the step size were 3.

\begin{figure}[!ht]
\centering
\includegraphics[scale=0.8]{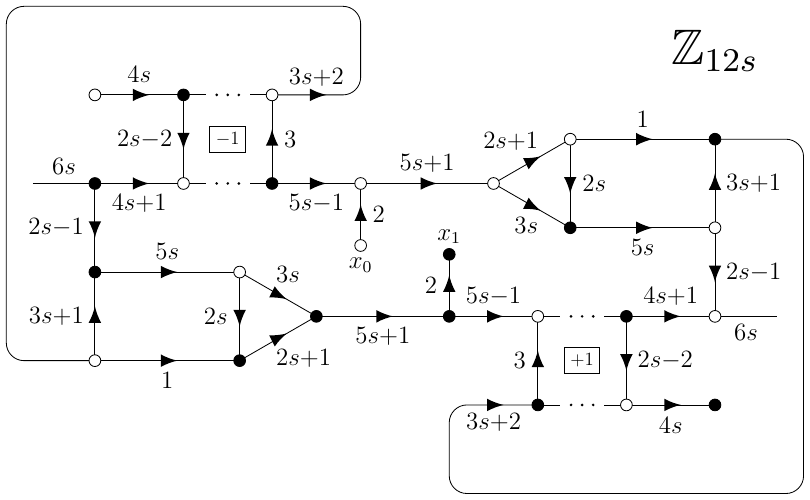}
\caption{$B_{12s+1}$, even $s \geq 2$.}
\label{fig-sc1-even}
\end{figure}

\begin{figure}[!ht]
\centering
\includegraphics[scale=0.8]{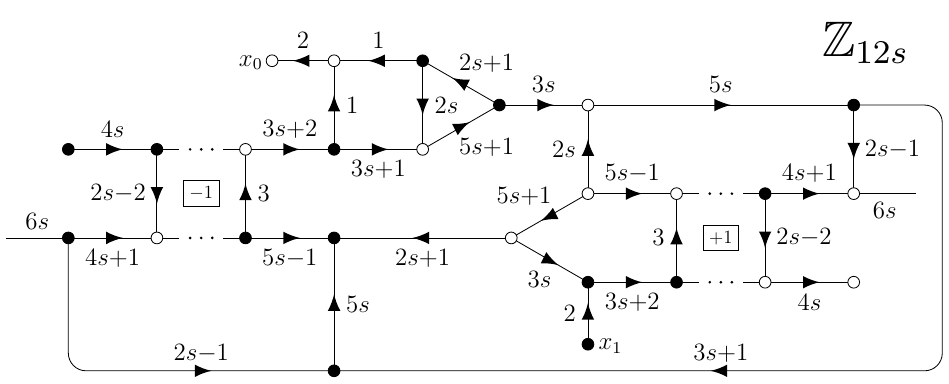}
\caption{$B_{12s+1}$, odd $s \geq 3.$}
\label{fig-sc1-odd}
\end{figure}

\section{Concluding Remarks}

We presented several families of index 2 current graphs and orientable cascades, the most important of which was for the octahedral graphs, confirming a conjecture of Jungerman and Ringel~\cite{JungermanRingel-Octa} that was open for over forty years. Plummer and Zha~\cite{PlummerZha-Connectivity, PlummerZha-Connectivity2} were interested in relationships between genus and connectivity, singling out the genus of these octahedral graphs as a special case that was still open. Besides the octahedral graphs, we may also ask about the genus of all graphs on $n$ vertices with minimum degree $n-2$, i.e., graphs where we delete a (not necessarily perfect) matching from a complete graph. Little is known for this family of graphs, even when the matching is of size 3 or 4. In Appendix~\ref{app-sporadic}, we found an orientable triangular embedding of $K_{15}-6K_2$ via computer search. This embedding and that of $K_{12}-C_{12}$ settles two missing cases in a characterization of Plummer and Zha~\cite[Thm 2.4]{PlummerZha-Connectivity} on the ``$g$-unique'' complete graphs:

\begin{proposition}
$K_{10}$ and $K_{12}-C_{12}$ are both graphs with vertex connectivity 9 and genus 4. $K_{14}$ and $K_{15}-6K_2$ are both graphs with vertex connectivity $13$ and genus $10$. Thus $K_{10}$ and $K_{14}$ are not unique with respect to these parameters. 
\end{proposition}

An independent proof of this fact, using other graphs, is also given by Bokal \emph{et al.}\ \cite{Bokal-Dual}.

\bibliographystyle{alpha}
\bibliography{biblio}

\appendix

\section{Embeddings of some sporadic cases}\label{app-sporadic}

All the embeddings presented here are orientable and triangular. To help with verifying the latter, we note that some of these embeddings are symmetric, i.e., of some index less than the number of vertices. 

The following embedding of $K_{10}-3K_2$ was found by Jonathan L.\ Gross (personal communication) by modifying the standard triangular embedding of $K_{10}-K_3$:
{\small$$\begin{array}{rrrrrrrrrrrrrrrrr}
0. & 1 & 2 & x & y & 5 & 4 & z & 6 & 3 \\
1. & 2 & 0 & 3 & y & 6 & 5 & z & 4 \\
2. & 3 & x & 0 & 1 & 4 & 6 & y & z & 5 \\
3. & 4 & x & 2 & 5 & y & 1 & 0 & 6 \\
4. & 5 & x & 3 & 6 & 2 & 1 & z & 0 \\
5. & 6 & x & 4 & 0 & y & 3 & 2 & z & 1 \\
6. & x & 5 & 1 & y & 2 & 4 & 3 & 0 & z \\
x. & 0 & 2 & 3 & 4 & 5 & 6 & z & y \\
y. & 0 & x & z & 2 & 6 & 1 & 3 & 5 \\
z. & 0 & 4 & 1 & 5 & 2 & y & x & 6 \\
\end{array}$$}

This embedding of $K_{12}-C_{12}$ is derived from the index 4 current graph in Figure~\ref{fig-h12}: 
{\small$$\begin{array}{rrrrrrrrrrrrrrrrr}
0. & 8 & 5 & 9 & 4 & 2 & 6 & 3 & 7 & 10 \\
1. & 5 & 8 & 4 & 9 & 3 & 10 & 6 & 11 & 7 \\
2. & 6 & 0 & 4 & 11 & 8 & 10 & 5 & 7 & 9 \\
3. & 7 & 0 & 6 & 8 & 11 & 5 & 10 & 1 & 9 \\
4. & 0 & 9 & 1 & 8 & 6 & 10 & 7 & 11 & 2 \\
5. & 9 & 0 & 8 & 1 & 7 & 2 & 10 & 3 & 11 \\
6. & 10 & 4 & 8 & 3 & 0 & 2 & 9 & 11 & 1 \\
7. & 11 & 4 & 10 & 0 & 3 & 9 & 2 & 5 & 1 \\
8. & 4 & 1 & 5 & 0 & 10 & 2 & 11 & 3 & 6 \\
9. & 1 & 4 & 0 & 5 & 11 & 6 & 2 & 7 & 3 \\
10. & 2 & 8 & 0 & 7 & 4 & 6 & 1 & 3 & 5 \\
11. & 3 & 8 & 2 & 4 & 7 & 1 & 6 & 9 & 5
\end{array}$$}
For more information on index 4 current graphs, see Pengelley and Jungerman~\cite{Pengelley-Index4} and Korzhik~\cite{Korzhik-Index4}. 

\begin{figure}[!ht]
\centering
\includegraphics[scale=0.8]{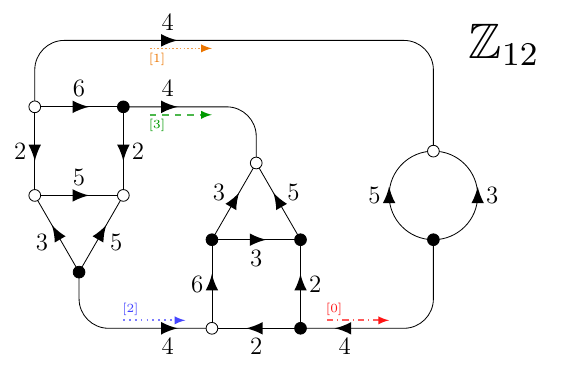}
\caption{An index 4 current graph generating a triangular embedding of $K_{12}-C_{12}$.}
\label{fig-h12}
\end{figure}

In addition to $K_{10}-K_3$, we need triangular embeddings of $O_{12}+\overline{K_4}$ and $O_{18}+\overline{K_4}$:

{\small$$\begin{array}{rrrrrrrrrrrrrrrrr}
0. & 1 & 2 & w & 3 & x & 4 & 5 & 7 & 8 & y & 9 & z & 10 & 11 \\
1. & 0 & 11 & 4 & 9 & 10 & 3 & w & 5 & y & 8 & z & 6 & x & 2 \\
2. & 0 & 1 & x & 7 & 6 & 4 & y & 5 & z & 9 & 11 & 3 & 10 & w \\
3. & 0 & w & 1 & 10 & 2 & 11 & y & 4 & z & 8 & 6 & 7 & 5 & x \\
4. & 0 & x & 8 & 7 & w & 9 & 1 & 11 & z & 3 & y & 2 & 6 & 5 \\
5. & 0 & 4 & 6 & z & 2 & y & 1 & w & 8 & 10 & 9 & x & 3 & 7 \\
6. & 1 & z & 5 & 4 & 2 & 7 & 3 & 8 & 9 & w & 10 & y & 11 & x \\
7. & 0 & 5 & 3 & 6 & 2 & x & 9 & y & 10 & z & 11 & w & 4 & 8 \\
8. & 0 & 7 & 4 & x & 10 & 5 & w & 11 & 9 & 6 & 3 & z & 1 & y \\
9. & 0 & y & 7 & x & 5 & 10 & 1 & 4 & w & 6 & 8 & 11 & 2 & z \\
10. & 0 & z & 7 & y & 6 & w & 2 & 3 & 1 & 9 & 5 & 8 & x & 11 \\
11. & 0 & 10 & x & 6 & y & 3 & 2 & 9 & 8 & w & 7 & z & 4 & 1 \\
w. & 0 & 2 & 10 & 6 & 9 & 4 & 7 & 11 & 8 & 5 & 1 & 3 \\
x. & 0 & 3 & 5 & 9 & 7 & 2 & 1 & 6 & 11 & 10 & 8 & 4 \\
y. & 0 & 8 & 1 & 5 & 2 & 4 & 3 & 11 & 6 & 10 & 7 & 9 \\
z. & 0 & 9 & 2 & 5 & 6 & 1 & 8 & 3 & 4 & 11 & 7 & 10 \\
\end{array}$$}

{\small$$\begin{array}{rrrrrrrrrrrrrrrrrrrrrrrrrr}
0. & 1 & 5 & 12 & 10 & 2 & 14 & 16 & 6 & 11 & y & 15 & z & 17 & 4 & 8 & 13 & 7 & w & 3 & x \\
1. & 0 & x & 16 & w & 12 & 7 & 3 & 11 & 17 & 15 & 13 & 6 & 14 & 2 & z & 4 & y & 8 & 9 & 5 \\
2. & 0 & 10 & 4 & 6 & 8 & 7 & 12 & 16 & 15 & 9 & w & 5 & x & 3 & 13 & y & 17 & z & 1 & 14 \\
3. & 0 & w & 14 & 9 & 11 & 1 & 7 & 5 & 15 & 8 & 4 & z & 6 & y & 10 & 16 & 17 & 13 & 2 & x \\
4. & 0 & 17 & 11 & w & 7 & x & 5 & 9 & 16 & 14 & 12 & 6 & 2 & 10 & 15 & y & 1 & z & 3 & 8 \\
5. & 0 & 1 & 9 & 4 & x & 2 & w & 16 & 11 & 13 & 15 & 3 & 7 & 17 & 10 & 6 & z & 8 & y & 12 \\
6. & 7 & 11 & 0 & 16 & 8 & 2 & 4 & 12 & 17 & y & 3 & z & 5 & 10 & 14 & 1 & 13 & w & 9 & x \\
7. & 6 & x & 4 & w & 0 & 13 & 9 & 17 & 5 & 3 & 1 & 12 & 2 & 8 & z & 10 & y & 14 & 15 & 11 \\
8. & 6 & 16 & 10 & 12 & 14 & 13 & 0 & 4 & 3 & 15 & w & 11 & x & 9 & 1 & y & 5 & z & 7 & 2 \\
9. & 6 & w & 2 & 15 & 17 & 7 & 13 & 11 & 3 & 14 & 10 & z & 12 & y & 16 & 4 & 5 & 1 & 8 & x \\
10. & 6 & 5 & 17 & w & 13 & x & 11 & 15 & 4 & 2 & 0 & 12 & 8 & 16 & 3 & y & 7 & z & 9 & 14 \\
11. & 6 & 7 & 15 & 10 & x & 8 & w & 4 & 17 & 1 & 3 & 9 & 13 & 5 & 16 & 12 & z & 14 & y & 0 \\
12. & 13 & 17 & 6 & 4 & 14 & 8 & 10 & 0 & 5 & y & 9 & z & 11 & 16 & 2 & 7 & 1 & w & 15 & x \\
13. & 12 & x & 10 & w & 6 & 1 & 15 & 5 & 11 & 9 & 7 & 0 & 8 & 14 & z & 16 & y & 2 & 3 & 17 \\
14. & 12 & 4 & 16 & 0 & 2 & 1 & 6 & 10 & 9 & 3 & w & 17 & x & 15 & 7 & y & 11 & z & 13 & 8 \\
15. & 12 & w & 8 & 3 & 5 & 13 & 1 & 17 & 9 & 2 & 16 & z & 0 & y & 4 & 10 & 11 & 7 & 14 & x \\
16. & 12 & 11 & 5 & w & 1 & x & 17 & 3 & 10 & 8 & 6 & 0 & 14 & 4 & 9 & y & 13 & z & 15 & 2 \\
17. & 12 & 13 & 3 & 16 & x & 14 & w & 10 & 5 & 7 & 9 & 15 & 1 & 11 & 4 & 0 & z & 2 & y & 6 \\
w. & 0 & 7 & 4 & 11 & 8 & 15 & 12 & 1 & 16 & 5 & 2 & 9 & 6 & 13 & 10 & 17 & 14 & 3 \\
x. & 0 & 3 & 2 & 5 & 4 & 7 & 6 & 9 & 8 & 11 & 10 & 13 & 12 & 15 & 14 & 17 & 16 & 1 \\
y. & 0 & 11 & 14 & 7 & 10 & 3 & 6 & 17 & 2 & 13 & 16 & 9 & 12 & 5 & 8 & 1 & 4 & 15 \\
z. & 0 & 15 & 16 & 13 & 14 & 11 & 12 & 9 & 10 & 7 & 8 & 5 & 6 & 3 & 4 & 1 & 2 & 17 \\
\end{array}$$}

Finally, the following rotation system is of $K_{15}-6K_2$:

{\small$$\begin{array}{rrrrrrrrrrrrrrrrrrrr}
0. & 1 & 9 & 14 & 10 & 6 & 12 & 7 & 13 & 3 & 11 & 2 & 4 & 8 & 5 \\
1. & 0 & 5 & 14 & 11 & 10 & 4 & 13 & 8 & 12 & 6 & 3 & 7 & 9 \\
2. & 0 & 11 & 8 & 14 & 13 & 5 & 7 & 3 & 9 & 6 & 10 & 12 & 4 \\
3. & 0 & 13 & 14 & 12 & 8 & 10 & 5 & 9 & 2 & 7 & 1 & 6 & 11 \\
4. & 0 & 2 & 12 & 5 & 11 & 7 & 14 & 6 & 9 & 13 & 1 & 10 & 8 \\
5. & 0 & 8 & 9 & 3 & 10 & 7 & 2 & 13 & 11 & 4 & 12 & 14 & 1 \\
6. & 0 & 10 & 2 & 9 & 4 & 14 & 8 & 13 & 7 & 11 & 3 & 1 & 12 \\
7. & 0 & 12 & 9 & 1 & 3 & 2 & 5 & 10 & 14 & 4 & 11 & 6 & 13 \\
8. & 0 & 4 & 10 & 3 & 12 & 1 & 13 & 6 & 14 & 2 & 11 & 9 & 5 \\
9. & 0 & 1 & 7 & 12 & 13 & 4 & 6 & 2 & 3 & 5 & 8 & 11 & 14 \\
10. & 0 & 14 & 7 & 5 & 3 & 8 & 4 & 1 & 11 & 13 & 12 & 2 & 6 \\
11. & 0 & 3 & 6 & 7 & 4 & 5 & 13 & 10 & 1 & 14 & 9 & 8 & 2 \\
12. & 0 & 6 & 1 & 8 & 3 & 14 & 5 & 4 & 2 & 10 & 13 & 9 & 7 \\
13. & 0 & 7 & 6 & 8 & 1 & 4 & 9 & 12 & 10 & 11 & 5 & 2 & 14 & 3 \\ 
14. & 0 & 9 & 11 & 1 & 5 & 12 & 3 & 13 & 2 & 8 & 6 & 4 & 7 & 10 \\
\end{array}$$}

\section{A genus embedding of $K_{18}$}\label{app-k18}

As mentioned earlier, Jungerman and Ringel~\cite{JungermanRingel-Octa} found triangular embeddings for some of the octahedral graphs, including the one on $18$ vertices. Remarkably, their orientable cascade for that case, depicted in Figure~\ref{fig-current-k18}, can be augmented with two handles that allow us to add the nine missing edges. The log of the current graph is of the form
$$\begin{array}{rrrrrrrrrrrrrrrrr}
\lbrack0\rbrack. & \dots & 12 & 6 & \dots & 16 & 3 & 7 & \dots
\end{array}$$
so there are faces of the form $[0, 6, 12]$ and $[3, 15, 9]$. Merging these faces with a handle and performing some edge flips adds six of the missing edges, as in Figure~\ref{fig-current-k18-add}. The same handle operation allows us to add the remaining missing edges $(2, 11), (5, 14), (8, 17)$, except after incrementing all the values in the diagram by $2$ and omitting the edge flips.  

\begin{figure}[!ht]
\centering
\includegraphics[scale=0.8]{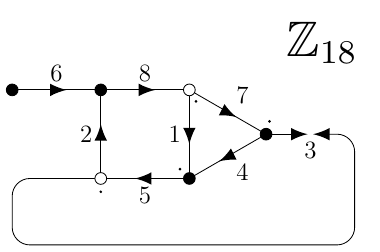}
\caption{Jungerman and Ringel's orientable cascade for $O_{18}$.}
\label{fig-current-k18}
\end{figure}

\begin{figure}[!ht]
\centering
\includegraphics[scale=0.8]{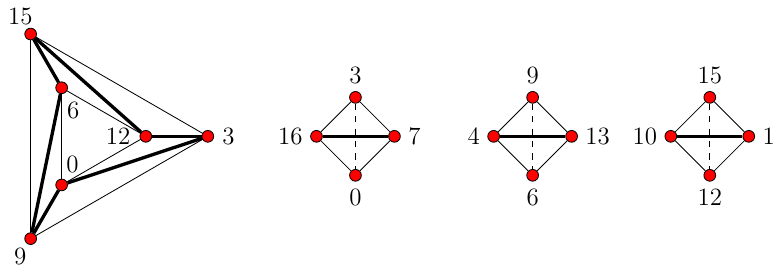}
\caption{A handle and edge flips for adding six missing edges.}
\label{fig-current-k18-add}
\end{figure}

Not only is this the first current graph-based construction of a genus embedding of $K_{18}$, but also that of a triangular embedding of $K_{18}$ minus three edges. Previous solutions for these cases~\cite{Mayer-Orientables, Jungerman-K18} were found essentially by trial and error, either by hand or by computer search.

\end{document}